\documentclass{amsart}
\usepackage{amsmath,amsthm,amsopn,amstext,amscd,amsfonts,amssymb,mathrsfs,mathtools}
\usepackage{dsfont}
\usepackage{comment}
\usepackage{xcolor}
\usepackage{float}
\usepackage[active]{srcltx}
\usepackage{subcaption}
\usepackage{graphicx,  mathdots}

  \usepackage{mathtools}

\newtheorem{theorem}{\sc Theorem}[section]

\newtheorem{eje}{\sc Example }[section]

\newtheorem{obs}{\sc Remark}[section]
\usepackage[active]{srcltx}

\newcommand{\dps}{\displaystyle}
\usepackage{lscape}
\usepackage{rotating}
\usepackage{caption}
\usepackage{multicol}
\usepackage{colortbl}
\usepackage{nicematrix}

\newmuskip\pFqskip
\mathchardef\pFcomma=\mathcode`, 

\makeatletter
\def\BState{\State\hskip-\ALG@thistlm}
\makeatother

\DeclareFontFamily{U}{mathb}{}
\DeclareFontShape{U}{mathb}{m}{n}{<-5.5> mathb5 <5.5-6.5> mathb6 
<6.5-7.5> mathb7 <7.5-8.5> mathb8 <8.5-9.5> mathb9 <9.5-11> mathb10 
<11-> mathb12}{}
\DeclareSymbolFont{mathb}{U}{mathb}{m}{n}
\DeclareFontSubstitution{U}{mathb}{m}{n}

\DeclareMathSymbol{\blackdiamond}{\mathbin}{mathb}{"0C}

\def\downbar#1{
\setbox10=\hbox{$#1$}
            \dimen10=\ht10 \advance\dimen10 by 2.5pt
            \ifdim \dimen10<15pt 
               \advance\dimen10 by -0.5pt
               \dimen11=\dimen10
               \advance\dimen10 by 2.5pt
               \lower \dimen11
            \else \lower \ht10 \fi
            \hbox {\hskip 1.5pt \vrule height \dimen10 depth \dp10}}
\def\upbar#1{
\setbox10=\hbox{$#1$}
            \dimen10=\ht10 \advance\dimen10 by \dp10 \advance\dimen10 by 2.5pt
            \ifdim \dimen10<15pt 
                \advance\dimen10 by 2pt \fi
            \raise 2.5pt \hbox {\hskip -1.5pt \vrule height \dimen10}}

\begin{document}
\title[On zeros of discrete POPUC]{On zeros of discrete paraorthogonal polynomials on the unit circle}

\author{G. Gordillo-Núñez}
\address{CMUC, Department of Mathematics, University of Coimbra, 3000-143 Coimbra, Portugal}
\email{up202310693@up.pt}
\author{A. Suzuki}
\address{CMUC, Department of Mathematics, University of Coimbra, 3000-143 Coimbra, Portugal}
\email{asuzuki@uc.pt}
\begin{abstract}
In this note we investigate, as  a natural continuation of [K. Castillo, Constr. Approx., 55 (2022) 605-627],  the behaviour of the zeros of discrete paraorthogonal polynomials on the unit circle with respect to a real parameter. 
\end{abstract}
\date{\today}
\maketitle
\section{Introduction}\label{intro}


Denote by $\mathbb{S}^1_r(c)$ the boundary of the open disk $\mathbb{D}_r(c)$ of radius $r>0$ and center at $c$, and write $C_r(c)=\mathbb{D}_r(c) \cap \mathbb{S}^1$ and $I_r(c)=\mathbb{D}_r(c)\cap \mathbb{R}$. For abbreviation, we use the notation $\mathbb{D}=\mathbb{D}_1(0)$ and $\mathbb{S}^1=\mathbb{S}^1_1(0)$. Let $\mathrm{d}\mu(\theta; t)$ be a finite nonnegative measure with infinite support (not a finite combination of point masses\footnote{The results of this epigraph are still valid, mutatis mutandis, when $\mathrm{d}\mu(\theta;t)$ has finite support. For instance, if $\mathrm{d}\mu(\theta;t)$ is a finite linear combination of $N$ distinct pure points, we only have OPUC up to the degree $N-1$.}) 
on $\mathbb{S}^1$ parametrized by $z=e^{i\theta}$ and depending on a parameter $t$ varying in a real open interval. We shall write simply $\mathrm{d}\mu$ when no confusion can arise. Let $(Q_n(\cdot; t))_{n\in\mathbb{N}}$ be the unique sequence of monic orthogonal polynomials on the unit circle (OPUC) associated with $\mathrm{d}\mu$, that is, polynomials $Q_n(z; t)=z^n+\cdots$ which satisfy
\begin{align*}
 \int Q_n(e^{i \theta}; t)\overline{Q_m(e^{i \theta}; t)}\,\mathrm{d}\mu(\theta; t)&=0 \quad (n\not=m=0,1,\dots),\\[7pt]
\int |Q_n(e^{i \theta}; t)|^2\mathrm{d}\mu(\theta; t)&\not=0.
\end{align*}
Fix $n \in \mathbb{N}\setminus\{0\}$. The corresponding monic paraorthogonal polynomial on the unit circle (POPUC) of degree $n+1$ associated with $\mathrm{d}\mu$ and a paraorthogonality parameter $b(t) \in \mathbb{S}^1$ is defined by (see \cite[p. $115$]{S11})
\begin{align}
\label{popuc} P_{n+1}(z; b(t); t)=P_{n+1}(z; t)=zQ_{n}(z; t)-\overline{b(t)} \,Q^*_{n}(z; t),
\end{align}
where $Q_n^*(z; t)=z^{n} \overline{Q_n(1/\overline{z}; t)}$. The POPUC satisfy
\begin{align}\label{exp}
\int  P_{n+1}(e^{i\theta}; t) \overline{g(e^{i\theta})} \,\mathrm{d}\mu(\theta; t)=0,
\end{align}
for any polynomial $g$ of degree at most $n$ vanishing at the origin. Assuming $P_{n+1}(\zeta(t); t)=0$, it is easy to see that (see \cite[(5)]{C22}) 
 \begin{align}\label{property}
\int  \frac{P_{n+1}(e^{i \theta}; t)}{e^{i \theta}-\zeta(t)}\,\, \overline{h(e^{i\theta})}\mathrm{d}\mu(\theta; t)=\overline{h(\zeta)} \int \frac{P_{n+1}(e^{i \theta}; t)}{e^{i \theta}-\zeta(t)} \mathrm{d}\mu(\theta; t),
\end{align}
for any nonzero polynomial $h$ of degree at most $n$.  In the numerical linear algebra the name POPUC is rarely used; however these polynomials were first discovered in this context (see \cite{C19} for historical comments). The two fundamental properties of the zeros of POPUC are the following (see \cite[Theorem $9.1.$]{G48}): (1) All the zeros of POPUC lie on $\mathbb{S}^1$; (2) The zeros of POPUC are all simple. It is important to emphasize that we always have the freedom to set a zero of a POPUC by appropriately choosing the paraorthogonality parameter. Indeed, given $\xi \in \mathbb{S}^1$ and  $\mathrm{d}\mu$, the corresponding normalized POPUC of degree $n+1$ with paraorthogonality parameter 
\begin{align*}
b(t)=\overline{\xi}\, \frac{\overline{Q_{n}(\xi; t)}}{\overline{Q^*_{n}(\xi; t)}}
\end{align*}
has a zero at $\xi$.

In applied problems, measures depending on one or more parameters are frequent. For instance, consider the measure
 \begin{align*}
\frac{|\Gamma(r+is+1)|^2}{\Gamma(2 r+1)}\,(2-2\cos \theta)^r\dps(-e^{i\theta})^{i s}\, \frac{\mathrm{d}\theta}{2\pi},
\end{align*}
for $r \in (-1/2, \infty)$, $s \in (-\infty, \infty)$, and $\theta \in [-\pi,\pi)$.  This measure belongs to a class of measures introduced by Fisher and Hartwig which has been the subject of numerous investigations $($see \cite{DIK11} and references therein$)$.  The behaviour of nonconstant zeros of POPUC associated with  $\mathrm{d}\mu(\theta; r, s)$ as $r$ and $s$ increase are studied in \cite{C22} as a consequence of the following:

\vspace{1mm}
\begin{changemargin}{0.8cm}{0.8cm} 
{\em Given a finite nonnegative measure $\omega(\theta; t)\, \mathrm{d}\mu(\theta)$ with infinite support on $\mathbb{S}^1$ parametrized by $z=e^{i\theta}$ $(\theta \in [\theta_0, \theta_0+2\pi))$ and depending on a parameter $t$ varying  in a real open interval, under suitable conditions, a nonconstant zero $\zeta(t)$ of a POPUC with a fixed zero at $e^{i\theta_0}$ moves strictly clockwise along $\mathbb{S}^1$ as $t$ increases on $I_\epsilon(t_0)$, provided that
 \begin{align*}
\frac{1}{\omega(\theta; t)}\frac{\partial \omega}{\partial t}(\theta; t)
\end{align*}
is a strictly increasing function of $\theta$ on  $(\theta_0, \theta_0+2\pi)$.}
\end{changemargin}
\vspace{1mm}

The results of \cite{C22} do not consider the possibility that the parameter appears in the discrete part of the measure. In fact, if the starting point is a measure formed by adding a finite linear combination of distinct pure points to $\mathrm{d}\mu(\theta;t)=\omega(\theta;t)\mathrm{d}\mu(\theta)$, i.e.,
\begin{align}\label{discrete}
\mathrm{d}\mu(\theta; t)+\sum_{j=0}^N \gamma_j(t)\, \delta(\theta-\omega_j(t)),
\end{align}
we cannot conclude anything about the behaviour of the zeros of the corresponding POPUC. Some properties of OPUC associated with measure perturbed by adding mass points have been studied by Simon (\cite[Section 10.13]{S05II} and references therein) and Wong (see \cite{W09}), among others. 
Moreover, if $\mathrm{d}\mu=0$ in \eqref{discrete}, we have (finite) discrete OPUC (and POPUC):
\begin{align*}
\sum_{j=0}^N  \gamma_j(t)\, Q_n(e^{i \omega_j(t)};t)\overline{Q_m(e^{i \omega_j(t)};t)}&=0 \quad(n\not=m=0,1,\dots, N-1),\\
\sum_{j=0}^N  \gamma_j(t)\, |Q_n(e^{i \omega_j(t)};t)|^2&\not=0.
\end{align*}
This finite discrete OPUC plays a fundamental role in several problems. For instance, when a certain function is being approximated, the approximation is constructed as a projection based on a discrete set of points. Geronimo and Liechty estimated the absolute error in the approximation and the error can be expressed in terms of a finite system of discrete OPUC (see \cite[Section 1.1]{GL20}).

Let us consider two examples of POPUC whose behaviour of zeros is impossible to deduce with the current results in the literature.

\begin{eje}
Consider the measure formed by adding a point mass to the degree one Bernstein-Szeg\H{o} measure \cite[Example 1.6.2]{S05I}$:$ 
$$
\frac{1-|\lambda|^2}{|1-\lambda e^{i \theta}|}\,\frac{\mathrm{d}\theta}{2\pi}+\gamma\, \delta(\theta-\omega)\quad (\lambda \in \mathbb{D},\, \gamma \in (0, \infty),\, \omega \in [0, 2\pi)).
$$
We leave it to the reader to verify that the OPUC associated with the above measure are given by 
\begin{align*}
&Q_n(z; \lambda; \gamma,\omega)=z^n-\overline{\lambda} z^{n-1}\\[7pt]
&\quad -\gamma\dfrac{e^{i(n-1)\omega}(e^{i\omega}-\overline{\lambda})}{1+\gamma\left(1+(n-1)\dfrac{|e^{i\omega}-\overline{\lambda}|^2}{1-|\lambda|^2}\right)}\left(1+\dfrac{(z-\overline{\lambda})(e^{-i\omega}-\lambda)(1-(e^{-i\omega}z)^{n-1})}{(1-|\lambda|^2)(1-e^{-i\omega}z)}\right).
\end{align*}
Define
\begin{align*}
P_{n+1}(z; \gamma)&=zQ_{n}(z;-i/3;\gamma,2\pi/3)-i\dfrac{Q_{n}(i;-i/3;\gamma,2\pi/3)}{Q_{n}^{*}(i;-i/3;\gamma,2\pi/3)}Q^*_{n}(z;-i/3;\gamma,2\pi/3),\\[7pt]
P_{n+1}(z; \omega)&=zQ_{n}(z;-i/3;1,\omega)-i\dfrac{Q_{n}(i;-i/3;1,\omega)}{Q_{n}^{*}(i;-i/3;1,\omega)}Q^*_{n}(z;-i/3;1,\omega).
\end{align*}
Figure \ref{fig0} shows the behaviour of the zeros of $P_{n+1}(z; \gamma)$ and $P_{n+1}(z; \omega)$ for different values of $\gamma$ and $\omega$, respectively. Note that both $P_5(z;\gamma)$ and $P_5(z;\omega)$ have a fixed zero at $z=i$. The zeros of $P_5(z;\gamma)$, except the fixed one, move strictly counterclockwise.

\begin{figure}[H]
  \centering
  \begin{subfigure}[b]{0.45\linewidth}
    \includegraphics[width=\linewidth]{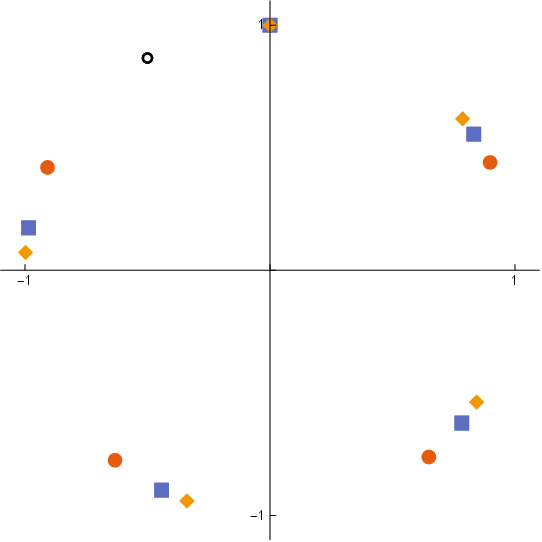}
  \end{subfigure}
  \begin{subfigure}[b]{0.45\linewidth}
    \includegraphics[width=\linewidth]{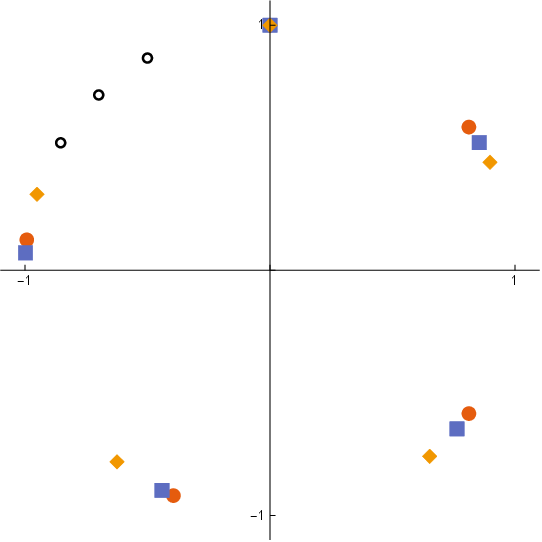}
   
  \end{subfigure}
  \caption{Zeros of $P_{5}(z; \gamma)$ (left plot)  and $P_{5}(z; \omega)$ (right plot)  for $\gamma$ equals $0.01$  (orange discs), $0.5$ (blue squares) and $5$ (yellow diamonds), and $\omega$ equals $2\pi/3$ (orange discs), $2\pi/3+1/4$ (blue squares), and $2\pi/3+1/2$ (yellow diamonds), respectively. The open discs represent the location of the mass points.}
  \label{fig0}
\end{figure}
\end{eje}

\begin{eje}
Consider the probability measure formed by adding a point mass to the Lebesgue measure \cite[Example 1.6.3]{S05I}
$$
(1-\gamma)\dfrac{\mathrm{d}\theta}{2\pi}+\gamma\delta(\theta-0)\quad (\gamma\in(0,1)).
$$
It follows easily that the POPUC associated with the above measure for a paraorthogonality parameter $b\in\mathbb{S}^1$ is given by
\begin{align*}
P_{n+1}(z;b;\gamma)&=z^{n+1}-\overline{b}-\dfrac{(1-\overline{b})\gamma}{1+(n-1)\gamma}\dps\sum_{j=1}^{n}z^j.
\end{align*}
Figure \ref{fig2} shows the behaviour of the zeros of $P_{n+1}(z;i;\gamma)$ and $P_{n+1}(z;-1;\gamma)$ for different values of $\gamma$. Note that $P_{5}(z;i;\gamma)$ and $P_{5}(z;-1;\gamma)$ have a fixed zero at $z=-1$ and $z=i$, respectively. 

\begin{figure}[H]
  \centering
  \begin{subfigure}[b]{0.45\linewidth}
    \includegraphics[width=\linewidth]{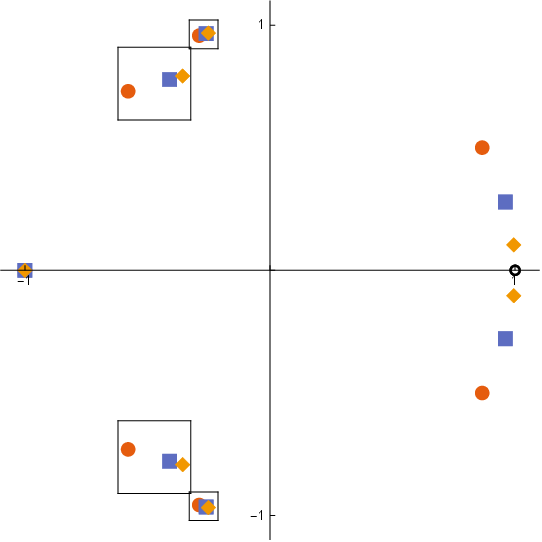}
  \end{subfigure}
  \begin{subfigure}[b]{0.45\linewidth}
    \includegraphics[width=\linewidth]{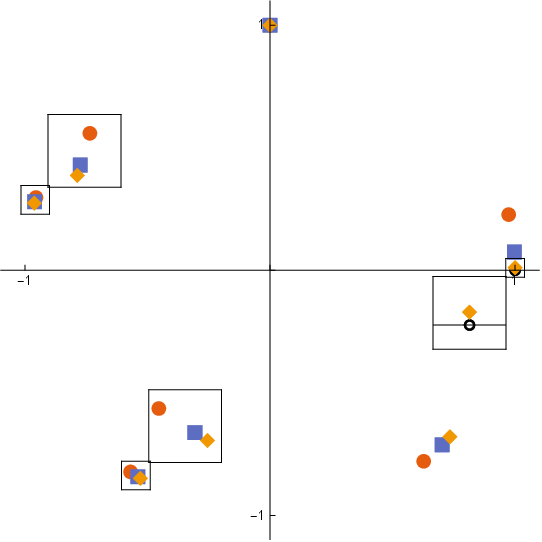}
   
  \end{subfigure}
  \caption{Zeros of $P_{5}(z;i;\gamma)$ (left plot)  and $P_{5}(z;-1;\gamma)$ (right plot)  for $\gamma$ equals $0.1$  (orange discs), $0.5$ (blue squares) and $0.9$ (yellow diamonds). The open discs represent the location of the mass points. Notice that some zeros have been zoomed in on for clarity.}
  \label{fig2}
\end{figure}

\end{eje}

In the next section we prove a ``discrete'' version of the results of \cite{C22} that will allow us to cover, in particular, the cases considered in the previous examples.
\section{Main results}\label{sec2}


We begin this section with the following result.

\begin{theorem}\label{main}
Let $$\mathrm{d}\mu(\theta; t)=\sum_{j=0}^{N}\gamma_j(t)\, \delta(\theta-\omega_j(t))$$ be a finite nonnegative measure on the unit circle parametrized by $z=e^{i\theta}$ and depending on a parameter $t$ varying  in a real open interval containing $t_0$. Suppose that $\gamma_j(t)$ and $\omega_j(t)$ are differentiable functions of $t$. Let $P(\cdot; t)$ be a nonconstant monic POPUC of degree $n$ associated with $\mathrm{d}\mu(\theta; t)$.  Assume that $P(\zeta_0; t_0)=0$. Then there exist $\epsilon>0$ and $\delta>0$, and a unique function, $\zeta: I_\epsilon(t_0) \to C_\delta(\zeta_0)$, differentiable on $I_\epsilon(t_0)$, with $\zeta(t_0)=\zeta_0$, and such that $P(\zeta(t); t)=0$ for each $t \in I_\epsilon(t_0)$. Assume that $P(e^{i\theta_0};t)=0$ for each $t\in  I_\epsilon(t_0)$, with $\zeta_0\neq e^{i\theta_0}$. Denote the zeros of $P(\cdot;t)$ by $\zeta_k(t)=e^{i\varphi_k(t)}$ $(k=1,\dots,n)$ and define
\begin{align}
s(\theta;t)&= \dfrac{\sin{\left(\dfrac{\varphi(t)-\theta_0}{2}\right)}}{2\sin{\left(\dfrac{\varphi(t)-\theta}{2}\right)}\sin{\left(\dfrac{\theta_0-\theta}{2}\right)}},\label{sf}\\[7pt]
S(\theta;t)&= \dps\sum_{k=1}^{n}\; \dfrac{1}{1+\delta_{k,n-1}+\delta_{k,n}}\, \cot{\left(\dfrac{\varphi_k(t)-\theta}{2}\right)},\label{Sf}
\end{align}
where we have set $\theta_0=\varphi_{n-1}(t)$ and $\varphi(t) = \varphi_{n}(t)$. Then $\zeta(t)=e^{i\varphi(t)}$ moves strictly counterclockwise along $\mathbb{S}^1$ as $t$ increases on $I_\epsilon(t_0)$, provided that
\begin{align*}
W_j(t) = s(\omega_j(t);t)\,\dfrac{\mathrm{d}\gamma_j}{\mathrm{d}t}(t)-\gamma_j(t)s(\omega_j(t);t)S(\omega_j(t);t)\dfrac{\mathrm{d}\omega_j}{\mathrm{d}t}(t)
\end{align*}
is nonnegative for all $j$ and nonzero for at least one value of $j$, and $\omega_j(t) \neq \varphi_k(t)$ for all $j,k$. Furthermore, if $W_j(t)=0$ for all $t\in I_\epsilon(t_0)$ and all $j$, then $\zeta(t)$ does not move as $t$ increases on $I_{\epsilon}(t_0)$.

\end{theorem}

\begin{proof}
The first part of the theorem follows from the proof of \cite[Proposition 3.1]{C22} and the details are left to the reader.

Assume $P$ has fixed degree $n\geq2$ and write $P_n$ instead of $P$. Set $\xi=e^{i\theta_0}$. 
Following the proof of \cite[Theorem 3.1]{C22} and using \cite[(15)]{C22}, we obtain, after elementary computations,
\begin{align}\label{C}
\dfrac{\mathrm{d}\varphi}{\mathrm{d}t}(t)
&=\dps\frac{\dps \sum_{j=0}^{N}\gamma_j(t)\dfrac{ie^{i\omega_j(t)}(\xi-\zeta(t))P_n(e^{i\omega_j(t)};t)}{(e^{i\omega_j(t)}-\xi)(e^{i\omega_j(t)}-\zeta(t))}\left.\overline{\dfrac{\partial P_n}{\partial t}(e^{i\theta};t)}\;\right|_{\theta=\omega_j(t)}}{\dps \int\left|\dfrac{P_n(e^{i\theta};t)}{e^{i\theta}-\zeta(t)}\right|^2\mathrm{d}\mu(\theta;t)},
\end{align}
with the denominator on the right-hand side being positive.
Let $\tau\in I_{\epsilon}(t_0)$. Since
$$
\dfrac{zP_n(z;\tau)}{(z-\xi)(z-\zeta(\tau))}
$$
is a nonzero polynomial of degree $n-1$ vanishing at the origin, we have, by \eqref{exp},
\begin{align}\label{po}
\dps\sum_{j=0}^{N}\gamma_j(t)\dfrac{e^{i\omega_j(t)}P_n(e^{i\omega_j(t)};\tau)}{(e^{i\omega_j(t)}-\xi)(e^{i\omega_j(t)}-\zeta(\tau))}\overline{P_n(e^{i\omega_j(t)};t)}=0.
\end{align}
Now, taking partial derivative of \eqref{po} with respect to $t$ yields
\begin{align}\label{pdpo}
&\dps\sum_{j=0}^{N}\gamma_j(t)\dfrac{e^{i\omega_j(t)}P_n(e^{i\omega_j(t)};\tau)}{(e^{i\omega_j(t)}-\xi)(e^{i\omega_j(t)}-\zeta(\tau))}\overline{\dfrac{\mathrm{d}P_n}{\mathrm{d} t}(e^{i\omega_j(t)};t)}\\[7pt]
& =-\dps\sum_{j=0}^{N}\left(\dfrac{\mathrm{d}\gamma_j}{\mathrm{d}t}(t)\dfrac{e^{i\omega_j(t)}P_n(e^{i\omega_j(t)};\tau)}{(e^{i\omega_j(t)}-\xi)(e^{i\omega_j(t)}-\zeta(\tau))}\overline{P_n(e^{i\omega_j(t)};t)}\right.\nonumber\\[7pt]
&\qquad \qquad \qquad \left.+\gamma_j(t)\dfrac{\partial}{\partial t}\left(\dfrac{e^{i\omega_j(t)}P_n(e^{i\omega_j(t)};\tau)}{(e^{i\omega_j(t)}-\xi)(e^{i\omega_j(t)}-\zeta(\tau))}\right)\overline{P_n(e^{i\omega_j(t)};t)}\right).\nonumber
\end{align}
In addition, note that
\begin{align*}
\dfrac{e^{i\omega_j(t)}P_n(e^{i\omega_j(t)};\tau)}{(e^{i\omega_j(t)}-\xi)(e^{i\omega_j(t)}-\zeta(\tau))}=e^{i\omega_j(t)}\dps\prod_{k=1}^{n-2}(e^{i\omega_j(t)}-\zeta_k(\tau)).
\end{align*}
Hence,
\begin{align}\label{ds1}
&\dfrac{\partial}{\partial t}\left(\dfrac{e^{i\omega_j(t)}P_n(e^{i\omega_j(t)};\tau)}{(e^{i\omega_j(t)}-\xi)(e^{i\omega_j(t)}-\zeta(\tau))}\right)
\\[7pt]
&=i\dfrac{\mathrm{d}\omega_j}{\mathrm{d}t}(t)\dfrac{e^{i\omega_j(t)}P_n(e^{i\omega_j(t)};\tau)}{(e^{i\omega_j(t)}-\xi)(e^{i\omega_j(t)}-\zeta(\tau))}\left(1+\dps\sum_{k=1}^{n-2}\dfrac{e^{i\omega_j(t)}}{e^{i\omega_j(t)}-\zeta_k(\tau)}\right).\nonumber
\end{align}
Furthermore,
\begin{align}\label{ds2}
\dfrac{\mathrm{d}P_n}{\mathrm{d} t}(e^{i\omega_j(t)};t)
&=\left.\dfrac{\partial P_n}{\partial t}(e^{i\theta};t)\right|_{\theta=\omega_j(t)}+i\dfrac{\mathrm{d}\omega_j}{\mathrm{d}t}(t)\dps\sum_{k=1}^{n}\dfrac{e^{i\omega_j(t)}P_n(e^{i\omega_j(t)};t)}{e^{i\omega_j(t)}-\zeta_k(t)}.
\end{align}
Substituting  \eqref{ds1} and \eqref{ds2} into \eqref{pdpo}, we can assert that
\begin{align}\label{pdpods}
&-\dps\sum_{j=0}^{N}\gamma_j(t)\dfrac{e^{i\omega_j(t)}P_n(e^{i\omega_j(t)};\tau)}{(e^{i\omega_j(t)}-\xi)(e^{i\omega_j(t)}-\zeta(\tau))}\left.\overline{\dfrac{\partial P_n}{\partial t}(e^{i\theta};t)}\right|_{\theta=\omega_j(t)}\\[7pt]
=&\dps\sum_{j=0}^{N}\dfrac{e^{i\omega_j(t)}P_n(e^{i\omega_j(t)};\tau)\overline{P_n(e^{i\omega_j(t)};t)}}{(e^{i\omega_j(t)}-\xi)(e^{i\omega_j(t)}-\zeta(\tau))}\nonumber\\[7pt]
& \quad \! \times \! \left(\!\dfrac{\mathrm{d}\gamma_j}{\mathrm{d}t}(t)\!+\!i\gamma_j(t)\dfrac{\mathrm{d}\omega_{j}}{\mathrm{d}t}(t)\!\left(1\!-\!\dps\sum^{n}_{k=1}\dfrac{e^{-i\omega_j(t)}}{e^{-i\omega_j(t)}-\overline{\zeta_k(t)}}\!+\!\dps\sum^{n-2}_{k=1}\dfrac{e^{i\omega_j(t)}}{e^{i\omega_j(t)}-\zeta_k(t)}\right)\!\!\right).\nonumber
\end{align}
Multiplying \eqref{pdpods} by $i(\xi-\zeta(\tau))$ and letting $\tau\rightarrow t$ yields
\begin{align}\label{pdpodsl}
&\dps\sum_{j=0}^{N}\gamma_j(t)\dfrac{ie^{i\omega_j(t)}(\xi-\zeta(t))P_n(e^{i\omega_j(t)};t)}{(e^{i\omega_j(t)}-\xi)(e^{i\omega_j(t)}-\zeta(t))}\left.\overline{\dfrac{\partial P_n}{\partial t}(e^{i\theta};t)}\right|_{\theta=\omega_j(t)}\\[7pt]
=&\dps\sum_{j=0}^{N}\dfrac{ie^{i\omega_j(t)}(\zeta(t)-\xi)\left|P_n(e^{i\omega_j(t)};t)\right|^2}{(e^{i\omega_j(t)}-\xi)(e^{i\omega_j(t)}-\zeta(t))}\nonumber\\[7pt]
& \quad\! \times \!\left(\!\dfrac{\mathrm{d}\gamma_j}{\mathrm{d}t}(t)\!+\!i\gamma_j(t)\dfrac{\mathrm{d}\omega_{j}}{\mathrm{d}t}(t)\!\left(1\!-\!\dps\sum^{n}_{k=1}\dfrac{e^{-i\omega_j(t)}}{e^{-i\omega_j(t)}-\overline{\zeta_k(t)}}\!+\!\dps\sum^{n-2}_{k=1}\dfrac{e^{i\omega_j(t)}}{e^{i\omega_j(t)}-\zeta_k(t)}\right)\!\right).\nonumber
\end{align}
It is straightforward to check that
\begin{align}
s(\theta;t)&=\dfrac{i(\zeta(t)-\xi)e^{i\theta}}{(e^{i\theta}-\xi)(e^{i\theta}-\zeta(t))}, \label{s1'}\\[7pt]
S(\theta;t)&=-i\left(1-\dps\sum^{n}_{k=1}\dfrac{e^{-i\theta}}{e^{-i\theta}-\overline{\zeta_k(t)}}+\dps\sum^{n-2}_{k=1}\dfrac{e^{i\theta}}{e^{i\theta}-\zeta_k(t)}\right).\label{s2'}
\end{align}
Hence, combining \eqref{C} and \eqref{pdpodsl},
\begin{align}\label{C1}
\dfrac{\mathrm{d}\varphi}{\mathrm{d}t}(t)\dps \int\left|\dfrac{P_n(e^{i\theta};t)}{e^{i\theta}-\zeta(t)}\right|^2\mathrm{d}\mu(\theta;t)& =\dps\sum_{j=0}^{N}W_j(t)\left|P_n(e^{i\omega_j(t)};t)\right|^2,
\end{align}
and the second part of the theorem follows. The last part of the theorem follows directly from \eqref{C1}. 

\end{proof}

We also obtain a similar result to Theorem \ref{main} whenever a POPUC has a pair of complex conjugate zeros.

\begin{theorem}\label{conjugatezeros}
Assume the hypotheses and notation of Theorem \ref{main} and its proof, except that $P(e^{i\theta_0}; t)=0$. 
Suppose that $P(\overline{\zeta(t)}; t)=0$ and $\varphi(t)\in (0,\pi)$ $({\rm mod[-\pi, \pi)})$  for each $t \in I_\epsilon(t_0)$. Denote the zeros of $P(\cdot;t)$ by $\zeta_k(t)=e^{i\varphi_k(t)}$ $(k=1,\dots,n)$ and define
\begin{align}
\widetilde{s}(\theta;t)& = \dfrac{1}{2}\dfrac{1}{\cos{(\varphi(t))}-\cos(\theta)}, \label{sconj}\\
\widetilde{S}(\theta;t)& = \dfrac{\sin{(\theta)}}{\cos{(\theta)}-\cos(\varphi(t))}+\dps\sum_{k=1}^{n-2}\cot{\left(\dfrac{\varphi_k(t)-\theta}{2}\right)},\label{Sconj}
\end{align}
where we have set $\varphi(t) = \varphi_{n}(t)=-\varphi_{n-1}(t)$. Then $\zeta(t)=e^{i\varphi(t)}$ moves strictly counterclockwise along $\mathbb{S}^1$ as $t$ increases on $I_\epsilon(t_0)$, provided that
\begin{equation*}
\widetilde{W}_j(t) = \widetilde{s}(\omega_j(t);t)\,\dfrac{\mathrm{d}\gamma_j}{\mathrm{d}t}(t)-\gamma_j(t)\widetilde{s}(\omega_j(t);t)\widetilde{S}(\omega_j(t);t)\dfrac{\mathrm{d}\omega_j}{\mathrm{d}t}(t)
\end{equation*}
is nonnegative for all $j$ and nonzero for at least one value of $j$, and $\omega_j(t) \neq \varphi_k(t)$ for all $j,k$. Furthermore, if $\widetilde{W}_j(t)=0$ for all $t\in I_\epsilon(t_0)$ and all $j$, then $\zeta(t)$ does not move as $t$ increases on $I_{\epsilon}(t_0)$.

\end{theorem}

\begin{proof}
Again, we can follow the proof of \cite[Theorem 3.2]{C22} and, using \cite[(22)]{C22}, write
\begin{equation}\label{czdphi}
C(t)\dfrac{\mathrm{d}\varphi}{\mathrm{d}t}(t)=2\Im(\zeta(t))
\dps\sum_{j=0}^{N}\dfrac{\gamma_j(t)e^{i\omega_j(t)}P_n(e^{i\omega_j(t)};t)}{(e^{i\omega_j(t)}-\zeta(t))(e^{i\omega_j(t)}-\overline{\zeta(t)})}\left.\overline{\dfrac{\partial P_n}{\partial t}(e^{i\theta};t)}\;\right|_{\theta=\omega_j(t)},
\end{equation}
where
\begin{equation}\label{Cconjugate}
C(t)=\int\left|\dfrac{P_n(e^{i\theta};t)}{e^{i\theta}-\zeta(t)}\right|^2\mathrm{d}\mu(\theta;t)+\int\left|\dfrac{P_n(e^{i\theta};t)}{e^{i\theta}-\overline{\zeta(t)}}\right|^2\mathrm{d}\mu(\theta;t)>0.
\end{equation}
Replacing $\xi$ by $\overline{\zeta(t)}$ in   \eqref{pdpodsl}, we can assert
\begin{align}\label{czpdpodsl}
&\dps\sum_{j=0}^{N}\gamma_j(t)\dfrac{ie^{i\omega_j(t)}(\overline{\zeta(t)}-\zeta(t))P_n(e^{i\omega_j(t)};t)}{(e^{i\omega_j(t)}-\overline{\zeta(t)})(e^{i\omega_j(t)}-\zeta(t))}\left.\overline{\dfrac{\partial P_n}{\partial t}(e^{i\theta};t)}\right|_{\theta=\omega_j(t)}\\[7pt]
=&\dps\sum_{j=0}^{N}\dfrac{ie^{i\omega_j(t)}(\zeta(t)-\overline{\zeta(t)})\left|P_n(e^{i\omega_j(t)};t)\right|^2}{(e^{i\omega_j(t)}-\overline{\zeta(t)})(e^{i\omega_j(t)}-\zeta(t))}\nonumber\\[7pt]
& \quad \times\!\left(\!\dfrac{\mathrm{d}\gamma_j}{\mathrm{d}t}(t)\!+\!i\gamma_j(t)\dfrac{\mathrm{d}\omega_{j}}{\mathrm{d}t}(t)\!\left(1\!-\!\dps\sum^{n}_{k=1}\dfrac{e^{-i\omega_j(t)}}{e^{-i\omega_j(t)}-\overline{\zeta_k(t)}}\!+\!\dps\sum^{n-2}_{k=1}\dfrac{e^{i\omega_j(t)}}{e^{i\omega_j(t)}-\zeta_k(t)}\right)\!\right).\nonumber
\end{align}
Again, following    \eqref{s1'} and \eqref{s2'}, it is straightforward to check that
\begin{align*}
2\Im{(\zeta(t))}&=i(\overline{\zeta(t)}-\zeta(t)),\\[7pt]
\widetilde{s}(\theta;t)&=-\dfrac{e^{i\theta}}{(e^{i\theta}-\overline{\zeta(t)})(e^{i\theta}-\zeta(t))}, \\[7pt]
\widetilde{S}(\theta;t)&=-i\left(1-\dps\sum^{n}_{k=1}\dfrac{e^{-i\theta}}{e^{-i\theta}-\overline{\zeta_k(t)}}+\dps\sum^{n-2}_{k=1}\dfrac{e^{i\theta}}{e^{i\theta}-\zeta_k(t)}\right).
\end{align*}
Hence, combining \eqref{czdphi} and \eqref{czpdpodsl}, we deduce
\begin{align}\label{czC1}
C(t)\dfrac{\mathrm{d}\varphi}{\mathrm{d}t}(t)& =2\Im{(\zeta(t))}\dps\sum_{j=0}^{N}\widetilde{W}_j(t)\left|P_n(e^{i\omega_j(t)};t)\right|^2, \nonumber
\end{align}
and the result follows as in the end of the proof of Theorem \ref{main}.
\end{proof}

Finally, combining \cite[Theorem 3.1]{C22} and Theorem \ref{main}, we can also study measures that consist of both continuous and discrete components.

\begin{theorem}\label{cd}
Let
\begin{equation}
\label{dmu}
\mathrm{d}\mu(\theta;t)=\omega(\theta;t)\mathrm{d}\mu(\theta)+\dps\sum_{j=0}^{N}\gamma_j(t)\delta(\theta-\omega_j(t))
\end{equation}
be a finite nonnegative measure on the unit circle parametrized by $z=e^{i\theta}$ $(\theta \in [\theta_0, \theta_0+2\pi))$ and depending on a parameter $t$ varying  in a real open interval containing $t_0$. Suppose that $\gamma_j(t)$ and $\omega_j(t)$ are differentiable functions of $t$, and that, for almost all $\theta\in[\theta_0,\theta_0+2\pi)$, the function $\omega(\theta;t)$ is finite and admits partial derivative with respect to $t$, and that there exists a $\mu(\theta)$-integrable function $\alpha$ such that $$\left|\dfrac{\partial\omega}{\partial t}(\theta;t)\right|\leq\alpha(\theta),$$ almost everywhere in $[\theta_0,\theta_0+2\pi)$.  Let $P(\cdot; t)$ be a nonconstant monic POPUC associated with $\mathrm{d}\mu(\theta;t)$.  Assume that $P(\zeta_0; t_0)=0$. Then there exist $\epsilon>0$ and $\delta>0$, and a unique function, $\zeta: I_\epsilon(t_0) \to C_\delta(\zeta_0)$, differentiable on $I_\epsilon(t_0)$, with $\zeta(t_0)=\zeta_0$, and such that $P(\zeta(t); t)=0$ for each $t \in I_\epsilon(t_0)$. Assume that $P(e^{i\theta_0};t)=0$ for each $t\in  I_\epsilon(t_0)$, with $\zeta_0\neq e^{i\theta_0}$. Denote the zeros of $P(\cdot;t)$ by $\zeta_k(t)=e^{i\varphi_k(t)}$ $(k=1,\dots,n)$ and define
\begin{equation}
\label{otcont}
W(\theta;t) = s(\theta;t) \left(\dfrac{1}{\omega(\theta;t)}\dfrac{\partial\omega}{\partial t}(\theta;t)-\dfrac{1}{\omega(\varphi(t);t)}\left.\dfrac{\partial\omega}{\partial t}(\theta;t)\right|_{\theta=\varphi(t)}\right),
\end{equation}
and 
\begin{align}
W_j(t)= &s(\omega_j(t);t)\dfrac{\mathrm{d}\gamma}{\mathrm{d}t}(t)-\gamma_j(t)s(\omega_j(t);t)S(\omega_j(t);t)\dfrac{\mathrm{d}\omega_j}{\mathrm{d}t}(t)\label{otdisc}\\[3pt]
&-\gamma_j(t)s(\omega_j(t);t)\dfrac{1}{\omega(\varphi(t);t)}\left.\dfrac{\partial\omega}{\partial t}(\theta;t)\right|_{\theta=\varphi(t)}, \nonumber
\end{align}
where $s(\theta;t)$ and $S(\theta;t)$ are given by \eqref{sf} and \eqref{Sf}, respectively, and we have set $\theta_0=\varphi_{n-1}(t)$ and $\varphi(t) = \varphi_{n}(t)$. Then $\zeta(t)=e^{i\varphi(t)}$ moves strictly counterclockwise on $I_\epsilon(t_0)$, provided
$$
f(\theta;t)=\dfrac{1}{\omega(\theta;t)}\dfrac{\partial\omega}{\partial t}(\theta;t)
$$
is an increasing function of $\theta$ on $(\theta_0,\theta_0+2\pi)$ and $W_j(t)\geq0$ for all $t\in I_\epsilon(t_0)$ and all $j$, with at least one of the conditions being strict. Furthermore, if $W(\theta;t)=W_j(t)=0$ for all $t\in I_\epsilon(t_0)$ and for all $j$, then $\zeta(t)$ does not move as $t$ increases on $I_{\epsilon}(t_0)$.


\end{theorem}

\begin{proof}
Assume $P$ has fixed degree $n\geq2$ and write $P_n$ instead of $P$. The first part follows from the proof of \cite[Proposition 3.1]{C22} and the details are left to the reader.

Set $\xi=e^{i\theta_0}$. Following the proof of \cite[Theorem 3.1]{C22} and using \cite[(15)]{C22}, we obtain, after elementary computations,
\begin{align}\label{ccd}
C(t)\dfrac{\mathrm{d}\varphi}{\mathrm{d}t}(t)=&\int\dfrac{ie^{i\theta}(\xi-\zeta(t))P_n(e^{i\theta};t)}{(e^{i\theta}-\xi)(e^{i\theta}-\zeta(t))}\overline{\dfrac{\partial P_n}{\partial t}(e^{i\theta};t)}\omega(\theta;t)\mathrm{d}\mu(\theta)\\[7pt]
&+\dps\sum_{j=0}^{N}\gamma_j(t)\dfrac{ie^{i\omega_j(t)}(\xi-\zeta(t))P_n(e^{i\omega_j(t)};t)}{(e^{i\omega_j(t)}-\xi)(e^{i\omega_j(t)}-\zeta(t))}\left.\overline{\dfrac{\partial P_n}{\partial t}(e^{i\theta};t)}\right|_{\theta=\omega_j(t)},\nonumber
\end{align}
where
$$
C(t)=\dps \int\left|\dfrac{P_n(e^{i\theta};t)}{e^{i\theta}-\zeta(t)}\right|^2\mathrm{d}\mu(\theta;t)>0.
$$
\vspace{1mm}
Combining    \eqref{pdpodsl}, \eqref{ccd} and \cite[(17)]{C22} yields
\begin{align}\label{ccd1}
C(t)\dfrac{\mathrm{d}\varphi}{\mathrm{d}t}(t)&=\int s(\theta;t)\left|P_n(e^{i\theta};t)\right|^2\dfrac{\partial\omega}{\partial t}(\theta;t)\mathrm{d}\mu(\theta)\,\\[7pt]
&\!\quad+\!\dps\sum_{j=0}^{N}s(\omega_j(t);t)\!\left|P_n(e^{i\omega_j(t)};t)\right|^2\!\left(\!\dfrac{\mathrm{d}\gamma}{\mathrm{d}t}(t)\!-\!\gamma_j(t)\dfrac{\mathrm{d}\omega_j}{\mathrm{d}t}(t)S(\omega_j(t);t)\!\right)\!,\nonumber
\end{align}
where $s(\theta;t)$ and $S(\theta;t)$ are given by \eqref{s1'} and \eqref{s2'}, respectively. Now, from   \eqref{exp}, it is straightforward to check
\begin{align}\label{ccd2}
0=&\dps\int s(\theta;t)\left|P_n(e^{i\theta};t)\right|^2\dfrac{1}{\omega(\varphi(t);t)}\left(\left.\dfrac{\partial\omega}{\partial t}(\theta;t)\right|_{\theta=\varphi(t)}\right)\omega(\theta;t)\mathrm{d}\mu(\theta)\\[7pt]
&+\dps\sum_{j=0}^{N}\gamma_j(t)s(\omega_j(t);t)\left|P_n(e^{i\omega_j(t)};t)\right|^2\dfrac{1}{\omega(\varphi(t);t)}\left.\dfrac{\partial\omega}{\partial t}(\theta;t)\right|_{\theta=\varphi(t)}.\nonumber
\end{align}
Subtracting   \eqref{ccd2} from \eqref{ccd} yields
\begin{align}\label{ccd3}
C(t)\dfrac{\mathrm{d}\varphi}{\mathrm{d}t}(t)&=\!\int \!\left|P_n(e^{i\theta};t)\right|^2\!W(\theta;t)\omega(\theta;t)\mathrm{d}\mu(\theta)+\dps\sum_{j=0}^{N}|P_n(e^{i\omega_j(t)};t)|^2W_j(t).
\end{align}
Since $s(\theta;t)$ is negative for $\theta\in(\theta_0,\varphi(t))$ and positive for $\theta\in(\varphi(t),\theta_0+2\pi)$, and $f(\theta;t)$ is an increasing function of $\theta$ on $(\theta_0,\theta_0+2\pi)$, we conclude that $W(\theta;t)$ is nonnegative for $\theta\in(\theta_0,\varphi(t))\cup(\varphi(t),\theta_0+2\pi)$. Consequently, the integral on the right-hand side of \eqref{ccd3} is nonnegative (see \cite[Proof of Theorem 3.1]{C22}). By hypothesis, $W_j(t)\geq0$ for all $t\in I_{\epsilon}(t_0)$ and for all $j$, and, hence, the sum on the right-hand side of \eqref{ccd3} is nonnegative. 
Since at least one of the $W(\theta;t),W_0(t),\dots,W_N(t)$ is positive, it follows that the right-handside of \eqref{ccd3} is positive, and the second part of the theorem is proved. The last part of the theorem follows directly from \eqref{ccd3}.


\end{proof}

\begin{obs}\label{obsfix}
We can obtain some interesting results from Theorem \ref{cd} by assuming $W(\theta;t) = 0$ (which occurs, for example, when $\omega(\theta;t)$ is independent of $t$). Indeed, assuming the hypotheses and notation of Theorem \ref{cd} and $W(\theta;t) = 0$ for each $t\in I_{\epsilon}(t_0)$, we have the following:
\begin{enumerate}
\item Assume that $P_n(e^{i\omega_j(t)};t)=0$ for each $j=0,1,\dots,N$ and $t\in I_{\epsilon}(t_0)$, with $\omega_j(t)\neq\varphi(t)$ and $n>N+1$. Thus, the right-hand side of \eqref{ccd3} is equal to zero and, consequently,
$$
 \dfrac{\mathrm{d}\varphi_k}{\mathrm{d}t}(t) = 0,
$$
meaning that $\zeta(t)$ does not move as $t$ increases on $I_{\epsilon}(t_0)$.

\item Assume that the mass points are fixed, i.e.,
$$
\dfrac{\mathrm{d}\omega_j}{\mathrm{d}t}(t)=0,
$$
for each $j=0,1,\dots,N$ and $t\in I_{\epsilon}(t_0)$. Furthermore, assume that $n\geq N+1$ and $P_n(e^{i\omega_j(t)};t)=0$ for each $j=0,1,\dots,N$ and $t\in I_{\epsilon}(t_0)$. Hence, the right-hand side of \eqref{ccd3} is equal to zero and, consequently, $\zeta(t)$ does not move as $t$ increases on $I_{\epsilon}(t_0)$.
\end{enumerate}

\end{obs}

\section{Applications}

With the above results, we can now study the behaviour of the zeros of the POPUC considered in the introduction of this note.

\begin{eje}[Bernstein-Szeg\H{o}]
Let
$$
d\mu(\theta;\lambda;\gamma,\omega)=\frac{1-|\lambda|^2}{|1-\lambda e^{i \theta}|}\,\frac{\mathrm{d}\theta}{2\pi}+\gamma\, \delta(\theta-\omega),
$$
where $\lambda \in \mathbb{D}$, $\gamma \in (0, \infty)$, and $\omega \in [\theta_0,\theta_0+ 2\pi)$. 
We have $W(\theta;\gamma)=0$ and (see Figure \ref{fig4})
$$
W_0(\gamma)=\dfrac{\sin\left(\dfrac{\varphi(\gamma)-\theta_0}{2}\right)}{2\sin\left(\dfrac{\varphi(\gamma)-\omega}{2}\right)\sin\left(\dfrac{\theta_0-\omega}{2}\right)}
\begin{cases}
>0,\quad\varphi(\gamma)\in(\theta_0,\omega),\\[15pt]
<0,\quad\varphi(\gamma)\in(\omega,\theta_0+2\pi).
\end{cases}
$$
\begin{figure}[H]
\centering
\includegraphics[scale=0.7]{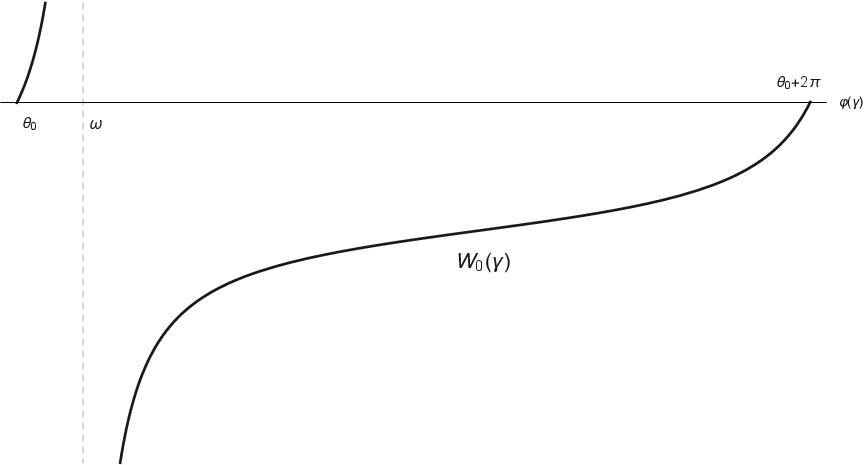}
\caption{The function $W_0$ for $\varphi(\gamma)\in[\theta_0,\theta_0+2\pi)$.}
\label{fig4}
\end{figure}
Therefore, by Theorem \ref{cd}, $\zeta(\gamma)$ moves strictly counterclockwise on $(\theta_0,\omega)$ and strictly clockwise on $(\omega,\theta_0+2\pi)$ as $\gamma$ increases (see Figure \ref{fig0}).
\end{eje}

\begin{eje}[Single inserted mass point]
Let
$$
\mathrm{d}\mu(\theta;\gamma)=(1-\gamma)\dfrac{\mathrm{d}\theta}{2\pi}+\gamma\delta(\theta-0),\qquad \gamma\in(0,1).
$$
We have $W(\theta;\gamma)=0$ and (see Figure \ref{fig1})
$$
W_0(\gamma)=\dfrac{1}{2(1-\gamma)}\dfrac{\sin\left(\dfrac{\varphi(\gamma)-\theta_0}{2}\right)}{\sin\left(\dfrac{\varphi(\gamma)}{2}\right)\sin\left(\dfrac{\theta_0}{2}\right)}
\begin{cases}
>0,\quad\varphi(\gamma)\in(\theta_0,2\pi),\\[15pt]
<0,\quad\varphi(\gamma)\in(2\pi,\theta_0+2\pi).
\end{cases}
$$
\begin{figure}[H]
\centering
\includegraphics[scale=0.7]{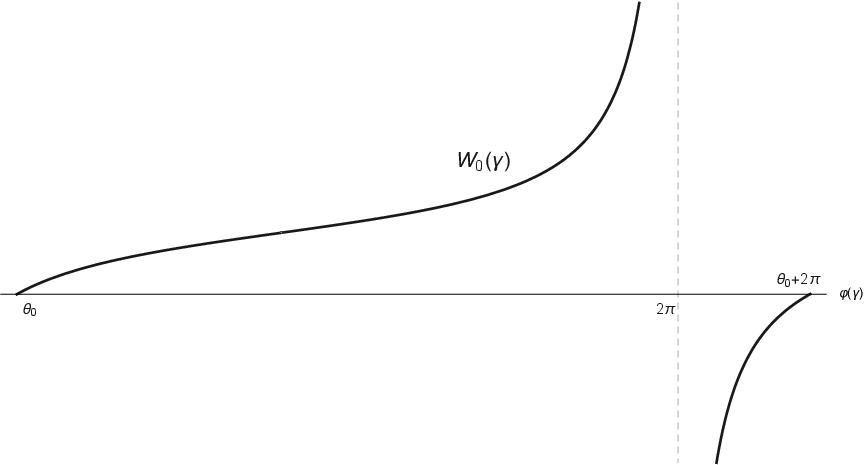}
\caption{The function $W_0$ for $\varphi(\gamma)\in[\theta_0,\theta_0+2\pi)$.}
\label{fig1}
\end{figure}
Therefore, by Theorem \ref{cd}, $\zeta(\gamma)$ moves strictly counterclockwise on $(\theta_0,2\pi)$ and strictly clockwise on $(2\pi,\theta_0+2\pi)$ as $\gamma$ increases, provided $\theta_0\in(0,2\pi)$ and $\varphi(\gamma)\in(\theta_0,\theta_0+2\pi)$ (see Figure \ref{fig2}).
\end{eje}

\section*{Acknowledgements}
The authors wish to express gratitude to their advisor Kenier Castillo for drawing attention to this problem, the continuous encouragement and constructive guidance.
This work is partially supported by the Centre for Mathematics of the University of Coimbra (funded by the Portuguese Government through FCT/MCTES, DOI 10.54499/UIDB/00324/2020). 

GGN is supported by the FCT grant UI.BD.154694.2023. AS is supported by the FCT grant 2021.05089.BD.

\bibliographystyle{plain}
\bibliography{bib} 
\end{document}